\documentclass[12pt]{amsart}
\usepackage[T2A]{fontenc}
\usepackage[cp1251]{inputenc}
\usepackage[english]{babel}
\usepackage{amsmath,latexsym,amsthm,amsfonts,tipa,upgreek}
\usepackage{amssymb,amscd,graphpap}

\newcommand\NN{{\mathbb N}}

\newcommand\QQ{{\mathbb Q}}
\newcommand\RR{{\mathbb R}}
\newcommand\w{{\omega}}
\newcommand\ee{{\varepsilon}}

\newcommand\BB{{\mathcal B}}

\newcommand\PP{{\mathcal P}}
\newcommand\FF{{\mathcal F}}

\newcommand\JJ{{\mathcal J}}
\newcommand\vt{{\Updelta}}
\newcommand\kk{{\varkappa}}

\newtheorem{Th}{Theorem}[section]

\newtheorem{Qs}{Question}[section]

\theoremstyle{definition}
\newtheorem{Cj}{Conjecture}[section]

\begin{document}

\title{Partitions of groups}
\author{Igor~Protasov, Sergii~Slobodianiuk}
\date{}
\maketitle

\begin{abstract} 
We classify the subsets of a group by their sizes, formalize the basic methods of partitions and apply them to
partition a group to subsets of prescribed sizes.
\end{abstract}

\section{Introduction}\label{s1}
In {\it Partition Combinatorics}, one of the main areas is {\em Ramsey Theory}. 
The motto of this theory could be "An absolute chaos is impossible": if a set $X$ is endowed with some structure
(say graph or semigroup) then under any finite partition of $X$ some cell of the partition has a good part of the hole structure.
For {\em Ramsey Theory of Groups} see, for examples, \cite{b17}, \cite{b24}.

On the other hand, in {\em Subset Combinatorics of Groups} (see survey \cite{b31}) we need in special partitions of a group into subsets of
pregiven sizes. A lot of concrete problems went from {\em Cromatic Combinatorics} \cite{b35}, 
Topological Algebra \cite[Chapter 13]{b14}, \cite{b25}, Asymptology \cite[Chapter 9]{b44}.
The main goal of this paper is to formalize the basic methods of partitions and demonstrate them in applications.

In Sections~\ref{s2} and~\ref{s3} we classify the diversity of subsets of a group by the sizes and give an 
ultrafilter characterization for each type of subsets.

In the next four sections we expose and apply the following methods of partitions: Grasshopper trick, Filtrations,
Submeasures and Three Sets Lemma.

In Sections~\ref{s8},~\ref{s9} we partition a group into thick and thin subsets and conclude the paper with some 
specific partitions in Section~\ref{s10}.

\section{Diversity of subsets}\label{s2}
Let $G$ be a group with the identity $e$ and let $\kk$ be a cardinal such that $\kk\le|G|$. 
We denote $[G]^{<\kk}=\{Y\subseteq X : |Y|<\kk\}$.

A subset $A$ of $G$ is called
\begin{itemize}
\item{} {\em left (right) $\kk$-large} if there exists $F\in [G]^{<\kk}$ such that $G=FA$ ($G=AF$);
\item{} {\em $\kk$-large} if $A$ is left and right $\kk$-large;
\item{} {\em left (right) $\kk$-small} if $L\setminus A$ is left (right) $\kk$-large for each left (right) $\kk$-large subset $L$;
\item{} {\em $\kk$-small} if $A$ is left and right $\kk$-small;
\item{} {\em left (right) $\kk$-thick} if for every $F\in [G]^{<\kk}$ there exists $a\in A$ such that $Fa\subseteq A$ ($aF\subseteq A$);
\item{} {\em $\kk$-thick} if $A$ is left and right $\kk$-thick;
\item{} {\em left (right) $\kk$-prethick} if there exists $F\in[G]^{<\kk}$ such that $FA$ is left $\kk$-thick ($AF$ is right $\kk$-thick );
\item{} {\em $\kk$-prethick} if $A$ is left and right $\kk$-prethick;
\end{itemize}

The left and right versions of above definitions are symmetric. 
For instance, $A$ is left $\kk$-large if and only if $A^{-1}$ is right $\kk$-large.
We note also that $A$ is left $\kk$-small if and only if $A$ is not left $\kk$-prethick, and $A$ is $\kk$-thick if and only if $L\cap A\neq\varnothing$ for every left $\kk$-large subset $L$.

In the case $\kk=\aleph_0$, we omit $\kk$ in the definitions. 
Thus, $A$ is left large if $G=FA$ for some finite subset $F$ of $G$.

We consider also some kinds of left small subsets omitting the adjective "left" in the following definitions.

A subset $A$ of $G$ is called
\begin{itemize}
\item{} {\em thin} if for every finite subset $F$ of $G$ such that $e\in F$, there exists a finite subset $H$ such 
that $Fa\cap A=\{a\}$ for every $a\in A\setminus H$;
\item{} {\em $n$-thin}, $n\in\NN$ if for every finite subset $F$ of $G$ such that $e\in F$, there exists a finite 
subset $H$ such that $|Fa\cap A|\le n$ for every $a\in A\setminus H$;
\item{} {\em sparse} if every infinite subset $S$ of $G$ contains a finite subset $F\subset S$ such that 
$\bigcap_{g\in F}gA$ is finite.
\end{itemize}

We note that a subset $A$ is thin if and only if, for every $g\in G\setminus\{e\}$, the subset $gA\cap A$ is finite.
Every $n$-thin subset is sparse.

We recall that a family $\JJ$ of subsets of a set $X$ is an {\em ideal} in the Boolean algebra $\PP_X$ of all subsets of $X$
if $X\notin\JJ$ and $A,B\in\JJ$, $C\subseteq A$ imply $A\cup B\in\JJ$ and $C\in\JJ$.
A family $\FF$ of subsets of $X$ is called a {\em filter} if the family $\{X\setminus A:A\in\FF\}$ is an ideal.

An ideal $\JJ$ of subsets of a group $G$ is called {\em left translation invariant} if $gA\in\JJ$ for all $g\in G$, $A\in\JJ$.

Now we fix a left translation invariant ideal $\JJ$ of subsets of $G$ and say that a subset $A$ is 
\begin{itemize}
\item{} {\em $\JJ$-large} if $G=FA\cup I$ for some $F\in[G]^{<\aleph_0}$ and $I\in\JJ$;
\item{} {\em $\JJ$-small} if $L\setminus A$ is $\JJ$-large for every $\JJ$-large subset $L$ of $G$;
\item{} {\em $\JJ$-thick} if $Int_F(A)\notin\JJ$ for each $F\in[G]^{<\aleph_0}$, where $Int_F(A)=\{a\in A:Fa\subseteq A\}$;
\item{} {\em $\JJ$-prethick} if $FA$ is $\JJ$-thick for some $F\in[G]^{<\aleph_0}$.
\end{itemize}
In the case $\JJ=\{\varnothing\}$, we get the above definitions of left large, left small, left thick, left prethick subsets.

We define also the relative versions of thin and sparse subsets. A subset $A$ of $G$ is called 
\begin{itemize}
\item{} {\em $\JJ$-thin} if $gA\cap A\in\JJ$ for every $g\in G\setminus\{e\}$;
\item{} {\em $\JJ$-sparse} if every infinite subset $S$ of $G$ contains a finite subset $F\subset S$ such that 
$\bigcap_{g\in F}gA\in\JJ$.
\end{itemize}

In the case $\JJ=[G]^{<\aleph_0}$, we get the above definitions of thin and sparse subsets. 

We conclude this section with much more general point of view at all these definitions.

A {\it ball structure} is a triple $\BB=(X,P,B)$, where $X$, $P$ are non-empty sets and, for every $x\in X$ and 
$\alpha\in P$, $B(x,\alpha)$ is a subset of $X$ which is called a {\em ball of radius $\alpha$ around $x$}.
It is supposed that $x\in B(x,\alpha)$ for all $x\in X$, $\alpha\in P$. 
The set $X$ is called the {\em support} of $\BB$, $P$  is called the {\em set of radii}.

Given any $x\in X$, $A\subseteq X$, $\alpha\in P$, we set
$$B^*(x,\alpha)=\{y\in X:x\in B(y,\alpha)\},\text{ } B(A,\alpha)=\bigcup_{a\in A}B(a,\alpha).$$

A ball structure $\BB=(X,P,B)$ is called a {\it ballean} if

\begin{itemize}
\item{} for any $\alpha,\beta\in P$, there exist $\alpha',\beta'$ such that, for every $x\in X$,
$$B(x,\alpha)\subseteq B^*(x,\alpha'),\ B^*(x,\beta)\subseteq B(x,\beta');$$

\item{} for any $\alpha,\beta\in P$, there exists $\gamma\in P$ such that, for every $x\in X$,
$$B(B(x,\alpha),\beta)\subseteq B(x,\gamma);$$

\item{} for any $x,y\in X$, there exists $\alpha\in P$ such that $y\in B(x, \alpha)$.
\end{itemize}

For a subset $Y\subseteq X$, we denote $\BB_Y(x,\alpha) = Y\cap B(x,\alpha)$.
A subset $Y$ is called {\em bounded} if $Y\subseteq B(x,\alpha)$ for some $x\in X$, $\alpha\in P$.

Given a ballean $\BB=(X,P,B)$, a subset $A$ of $X$ is called
\begin{itemize}
\item{} {\em large} if there exists $\alpha\in P$ such that $X=B(A,\alpha)$;
\item{} {\em small} if $L\setminus A$ large for every large subset $L$ of $X$;
\item{} {\em thick} if, for every $\alpha\in P$, there exists $a\in A$ such that $B(A,\alpha)\subseteq A$;
\item{} {\em prethick} if there exists $\beta\in P$ such that $B(A,\beta)$ is thick;
\item{} {\em thin} if, for every $\alpha$ there exists a bounded subset $Y$ of $X$ such that $B(a,\alpha)\cap A=\{a\}$
for each $a\in A\setminus Y$;
\item{} {\em sparse} if, for every unbounded subset $Y$ of $A$, there exists $\beta\in P$ such that, for every 
$\alpha\in P$, $B_A(y,\alpha)\setminus B_A(y,\beta)=\varnothing$ for some $y\in Y$;
\item{} {\em scattered} if, for every unbounded subset $Y$ of $A$, there exists $\beta\in P$ such that, for every 
$\alpha\in P$, $B_Y(y,\alpha)\setminus B_Y(y,\beta)=\varnothing$ for some $y\in Y$.
\end{itemize}

Let $G$ be an infinite group and let $\kk$ be an infinite cardinal, $\kk\le|G|$.
Given any $x\in X$ and $F\in[G]^{<\kk}$, we put $$B_l(x,F)=(F\cup\{e\})x\text{, }B_r(x,F)=x(F\cup\{e\})$$
and get two balleans $\BB_l(G,\kk)=(G,[G]^{<\kk},B_l)$, $\BB_r(G,\kk)=(G,[G]^{<\kk},B_r)$.

Clearly, a subset $A$ of $G$ is left $\kk$-large (left $\kk$-small, left $\kk$-thick, left $\kk$-prethick)
if and only if $A$ is large (small, thick, prethick) in the ballean $\BB_l(G,\kk)$.

We note that a subset $A$ of $G$ is thin (sparse) if and only if $A$ is thin (sparse) in the ballean $\BB_l(G,\aleph_0)$.

We say that a subset $A$ of a group $G$ is {\em scattered} if $A$ is scattered in the ballean $\BB_l(G,\aleph_0)$.

{\em Comments.} In the dynamical terminology \cite{b17}, left large and left prethick subsets are known under the 
names syndetic and piecewise syndetic.
The adjectives small, thick and thin in our context appeared in \cite{b9}, \cite{b10}, \cite{b11} respectively.
The sparse subsets were introduced in \cite{b15} and studied in \cite{b19}.

For $\JJ$-small, $\JJ$-thin and $\JJ$-sparse subsets see \cite{b2}, \cite{b3}, \cite{b21}, \cite{b38}.

The balleans can be considered as the asymptotic conterparts of uniform topological spaces (see \cite{b35}, \cite{b44}).
The balleans can also be defined in terms of entourages of the diagonal $\{(x,x):x\in X\}$. 
In this case, they are called {\em course structures} \cite{b45}.

In both contexts the large, small, thick, thin and scattered subsets of a ballean are counterparts of dense, nowhere
dense, open, discrete and scattered subspaces of a uniform space.

\section{Ultracompanions}\label{s3}
For a left invariant ideal $\JJ$ of subsets of an infinite group $G$, we consider a mapping $\Delta_\JJ:\PP_G\to\PP_G$ 
defined by $$\Delta_\JJ(A)=\{g\in G:gA\cap A\notin\JJ\}$$
and say that $\Delta_\JJ(A)$ is a {\em combinatorial derivation relatively the ideal} $\JJ$.
If $\JJ=[G]^{<\aleph_0}$, we omit $\JJ$ and say that $\Delta$ is a combinatorial derivation. In order to characterize 
the subsets of a group by their sizes, we need some ultrafilter versions of the combinatorial derivation.

We endow $G$ with the discrete topology and take the points of $\beta G$, the Stone-$\check{C}$ech compactification 
of $G$, to be the ultrafilters on $G$, with the points of $G$ identified with the principal ultrafilters on $G$. 
The topology on $\beta G$ can be defined by stating that the sets of the form $\overline{A}=\{p\in\beta G: A\in p\}$, 
where $A$ is a subset of $G$, form a base for the open sets. We note the sets of this form are clopen and that, for 
any $p\in\beta G$ and $A\subseteq G$, $A\in p$ if and only if $p\in\overline{A}$. 
We denote $A^*=\overline{A}\cap G^*$, where $G^*=\beta G\setminus G$. 

The universal property of $\beta G$ states that every mapping $f: G\to Y$, where $Y$ is a compact Hausdorff space, 
can be extended to the continuous mapping $f^\beta:\beta G\to Y$. We use this property to extend the group multiplication 
from $G$ to $\beta G$, see \cite[Chapter 4]{b17}, so $\beta G$ becomes a compact right topology semigroup.

Given a subset $A$ of a group $G$ and an ultrafilter $p\in G^*$ we define a {\em $p$-companion} of $A$ by
$$\vt_p(A)=A^*\cap Gp=\{gp: g\in G, A\in gp\},$$
and say that a subset $S$ of $G^*$ is an {\em ultracompanion} of $A$ if $S=\vt_p(A)$ for some $p\in G^*$.

To describe a relationship between ultracompanions and the combinatorial derivation, we denote 
$A_p=\{g\in G:A\in gp\}$ so $\vt_p(A)=A_p p$. Then $$\Delta_\JJ(A)=\bigcap\{A_p^{-1}:p\in \check{\JJ}, A\in p\}$$
where $\check{\JJ}=\{p\in G^*:X\setminus I\in p\text{ for each }p\in\JJ\}$. We observe that $\check{\JJ}$ is closed
in $G^*$ and $gp\in\check{\JJ}$ for all $g\in G$ and $p\in\check{\JJ}$.
\begin{Th}
For a subset $A$ of a group $G$ the following statements hold:

$(i)$ $A$ is $\JJ$-large if and only if $|\vt_p(A)|\leq1$ for each $p\in\check{\JJ}$.

$(ii)$ $A$ is $\JJ$-thick if and only if there exists $p\in\check{\JJ}$ such that $\vt_p(A)=Gp$;

$(iii)$ $A$ is $\JJ$-prethick if and only if there exists $p\in\check{\JJ}$ and $F\in[G]^{<\aleph_0}$ such that $\vt_p(FA)=Gp$;

$(iv)$ $A$ is $\JJ$-small if and only if, for every $p\in\check{\JJ}$ and each $F\in[G]^{<\aleph_0}$, we have $\vt_p(FA)\neq Gp$.
\end{Th}
\begin{Th}
A subset $A$ of $G$ is $n$-thin, $n\in\NN$ if and only if  $|\vt_p(A)|\le n$ for each $p\in G^*$.
\end{Th}
\begin{Th}
A subset $A$ of $G$ is sparse if and only if  $\vt_p(A)$ is finite for each $p\in G^*$.
\end{Th}

Let $(g_n)_{n\in\w}$ be an injective sequence  in a group $G$. The set 
$$\{g_{i_1}g_{i_2}\dots g_{i_n}:0\le i_1<i_2<\dots i_n<\w\}$$ is called an {\em $FP$-set} \cite[p. 406]{b17}.

Given a sequence $(b_n)_{n\in\w}$ in $G$, we say that the set 
$$\{g_{i_1}g_{i_2}\dots g_{i_n}b_{i_n}:0\le i_1<i_2<\dots i_n<\w\}$$ is a piecewise {\em shifted $FP$-set}.

\begin{Th} For a subset $A$ of a group $G$, the following statements are equivalent 
\begin{itemize}
\item[{(i)}] $A$ is scattered;
\item[{(ii)}] for every infinite subset $Y$ of $A$, there exists $p\in Y^*$ such that $\vt_p(Y)$ is finite;
\item[{(iii)}] $\vt_p(A)$ is discrete in $G^*$ for each $p\in G^*$;
\item[{(iv)}] $A$ contains no piecewise shifted $FP$-sets;
\end{itemize}\end{Th}
{\em Comments.} The combinatorial derivation was introduced in \cite{b33} and studied in \cite{b12}, \cite{b34}, \cite{b38}.

The results of this sections from \cite{b5}, \cite{b38}, \cite{b39}.
For ultracompanions of subsets of balleans see \cite{b8}.

\section{The grasshopper trick}\label{s4}
In \cite{b9}, Bella and Malykhin asked the following question: 
{\em can every infinite group $G$ be partitioned into two large subsets?}

The positive answer to this question was obtained in \cite{b27} (see also \cite[Chapter 3]{b35}) 
in the following strong form.
\begin{Th}\label{t4}
Every infinite group $G$ can be partitioned into $\aleph_0$ large subsets.
\end{Th}

To sketch the proof of this theorem, we consider two cases.\\
Case $1$: $G$ has an increasing chain $H_0\subset H_1\subset\dots\subset H_n\subset\dots$ of finite subgroups.\\
Case $2$: $G$ has an infinite finitely generating subgroup $H$.

In the first case, we use the Joint Transversal Theorem stating that given two partitions $\PP$ and $\PP'$ of a set 
$X$ such that $|P|=|P'|=n$, $n\in\NN$ for all $P\in\PP$ and $P'\in\PP'$ there exists a nonempty subset $T$ of $X$ 
such that $|T\cap P|=|T\cap P'|=1$ for all $P\in\PP$, $P'\in\PP'$.

In the second more delicate case, we use so-called {\em grasshopper trick}.
Let $\Gamma(V,E)$ be a connected graph with the set of vertices $V$ and the set of edges $E$.
We endow $V$ with the path metric: $d(u,v)$ is the length of a shortest path between $u$ and $v$.

If $\Gamma$ is finite, a numeration  $v_0, v_1, \dots , v_n$ of $V$ is called a {\em grasshopper cycle} if 
$d(v_0,v_1)\le3,\ d(v_1,v_2)\le3,\dots,\ d(v_{n-1},v_n)\le3,\ d(v_n,v_0)\le3$.
By the Grasshopper Theorem \cite[p. 26]{b35}, every finite graph admits a grasshopper cycle.

If $\Gamma$ is infinite, we say that an injective sequence $(v_n)_{n\in\w}$ in $V$ is a {\em grasshopper path} if
$d(v_i,v_{i+1})\le3$ for each $i<\w$. If there exists a grasshopper path coming through all vertices of $\Gamma$,
we say that $\Gamma$ is a {\em grasshopper ray}. If $\Gamma$ is countable, applying the Grasshopper Theorem, we can
partition, the set $V$ so that, for every cell $P$ of the partition, the induced graph $\Gamma[P]$ is a grasshopper ray.

Comming back to the second case, we take a finite symmetric system  $S$ of generators of $H$ and consider the Cayley
graph $Cay(G,S)$. The set of vertices of $Cay(G,S)$ is $G$, and the set of edges is 
$\{\{x,y\}:x,y\in G, x\neq y, x^{-1}y\in S\}$. Then we partition each connected components of  $Cay(G,S)$ into 
grasshopper rays and taking an appropriate countable partition of each grasshopper ray, we get a desired 
countable partition of $G$ into large subsets.

\section{Filtrations}\label{s5}
Let $G$ be an infinite group with the identity $e$, $\kk$ be an infinite cardinal.
A family $\{G_\alpha:\alpha<\kk\}$ of subgroups of $G$ is called a {\em filtration} if the following conditions hold
\begin{itemize}
\item[(1)] $G_0=\{e\}$ and $G=\bigcup_{\alpha<\kk}G_\alpha$;
\item[(2)] $G_\alpha\subset G_\beta$ for all $\alpha<\beta<\kk$;
\item[(3)] $\bigcup_{\alpha<\beta}G_\alpha=G_\beta$ for each limit ordinal $\beta<\kk$.
\end{itemize}
Clearly, a countable group $G$ admits a filtration if and only if $G$ is not finitely generated.
Every uncountable group $G$ of cardinality $\kk$ admits a filtration satisfying the additional condition $|G_\alpha|<\kk$ for each $\alpha<\kappa$.

For each $0<\alpha<\kappa$, we decompose $G_{\alpha+1}\setminus G_\alpha$ into right cosets by 
$G_\alpha$ and choose some system $X_\alpha$ of representatives so $G_{\alpha+1}\setminus G_\alpha=G_\alpha X_\alpha$.
We take an arbitrary element $g\in G\setminus\{e\}$ and choose the smallest subgroup $G_\alpha$ with $g\in G_\alpha$.
By $(3)$, $\alpha=\alpha_1+1$ for some ordinal $\alpha_1<\kk$. 
Hence, $g\in G_{\alpha+1}\setminus G_{\alpha_1}$ and there exist $g_1\in G_{\alpha_1}$ and $x_{\alpha_1}\in X_{\alpha_1}$ such that $g=g_1x_{\alpha_1}$.
If $g_1\neq e$, we choose the ordinal $\alpha_2$ and elements $g_2\in G_{\alpha_2+1}\setminus G_{\alpha_2}$ and $x_{\alpha_2}\in X_{\alpha_2}$ such that $g_1=g_2x_{\alpha_2}$.
Since the set of ordinals $\{\alpha:\alpha<\kk\}$ is well-ordered, after finite number $s(g)$ of steps, we get the representation
$$g=x_{\alpha_{s(g)}}x_{\alpha_{s(g)-1}}\dots x_{\alpha_2}x_{\alpha_1},\text{ }x_{\alpha_i}\in X_{\alpha_i}.$$
We note that this representation is unique and put
$$\gamma_1(g)=\alpha_1,\text{ }\gamma_2(g)=\alpha_2,\text{ }\dots,\gamma_{s(g)}(g)=\alpha_{s(g)}.$$

By Theorem~\ref{t4}, every infinite group $G$ can be partitioned into $\aleph_0$ $\aleph_0$-large subset.
If $G$ is amenable (in particular Abelian) and $|G|\ge\aleph_1$ then $G$ cannot be partitioned into $\aleph_1$
left $\aleph_0$-large subsets because $\mu(A)>0$ for each left $\aleph_0$-large subset of $G$ and every left 
invariant Banach measure $\mu$ on $G$.
\begin{Th}\label{t5.1}
Every infinite group $G$ of cardinality $\kk$ can be partitioned into $\kk$ left $\aleph_1$-large subsets.
\end{Th}
\begin{proof}
If $G$ is countable, the statement is evident because each singleton is $\aleph_1$-large.
Assume that $\kk>\aleph_0$ and fix some filtration $\{G_\alpha:\alpha<\kk\}$ of $G$ such that $|G_1|=\aleph_0$.
Given any $g\in G\setminus\{e\}$, we rewrite the canonical representation $g=x_{\gamma_n}\dots x_{\gamma_1},$
in the following form $$g=g_1x_{\gamma_m}\dots x_{\gamma_1},$$ $g_1\in G_1$, $0<\gamma_m<\dots<\gamma_1$. 
Here $g_1=e$ and $m=n$ if $\gamma_n>0$, and $g_1=x_{\gamma_n}$ and $m=n-1$ if $\gamma_n=0$.
We put $\Gamma(g)=\{\gamma_1,\dots,\gamma_m\}$ and fix an arbitrary bijection $\pi:G_1\to\NN$.

We define a family $\{A_\alpha:0<\alpha<\kk\}$ of subsets of $G$ by the following rule: 
$g\in A_\alpha$ if and only if $\alpha\in\Gamma(g)$ and $\gamma_{\pi(g_1)}=\alpha$. Since the subsets 
$\{A_\alpha:0<\alpha<\kk\}$ are pairwise disjoint, it suffices to show that each $A_\alpha$ is left $\aleph_1$-large.
We take $a_\alpha\in X_\alpha$, put $F_\alpha=\{e,a_\alpha\}G_1$ and prove that $G=F_\alpha A_\alpha$.

Let $g\in G$ and $\alpha\in \Gamma(g)$. By the definition of $A_\alpha$, there exists $h\in G_1$ such that 
$hg\in A_\alpha$ so $g\in G_1A_\alpha$. If $\alpha\notin\Gamma(g)$ then $\alpha\in\Gamma(a_\alpha^{-1}g)$ so
$a_\alpha^{-1}g\in G_1A_\alpha$ and $g\in a_\alpha G_1 A_\alpha$.
\end{proof}
\begin{Qs} Can every group $G$ of cardinality $\kk$ be partitioned into $\aleph_1$ $\aleph_1$-large subsets?
$\kk$-large subsets? $\kk^+$-large subsets ?\end{Qs}
Given a group $G$ and a subset $A$ of $G$, we denote
$$cov(A)=\min\{|X|:X\subseteq G, G=XA\}.$$
The covering number $cov(A)$ evaluates a size of $A$ inside $G$ and, if $A$ is a subgroup, coincides with the index $|G:A|$.

It is easy to partition each infinite group $G=A_1\cup A_2$ so that $cov(A_1)$ and $cov(A_2)$ are infinite. 
Moreover, if $|G|$ is regular, there is a partition $G=\bigcup\limits_{\alpha<|G|}H_\alpha$ such that 
$cov(G\setminus H_\alpha)=|G|$ for each $\alpha<|G|$. 
In particular, there is a partition $G=A_1\cup A_2$ such that $cov(A_1)=cov(A_2)=|G|$. 
See Section~\ref{s8} for these statements, their generalizations and applications.

However, for every $n\in\NN$, there is a (minimal) natural number $\Phi(n)$ such that, for every group $G$ and every 
partition $G=A_1\cup...\cup A_n$, $cov(A_iA_i^{-1})\le\Phi(n)$ for some cell $A_i$ of the partition. 
For these results and corresponding open problem see the next section.

In \cite[Question F]{b12}, J. Erde asked whether, given a partition $\PP$ of an infinite group $G$ such that 
$|\PP|<|G|$, there is $A\in\PP$ such that $cov(AA^{-1})$ is finite. 
After some simple examples answering this question extremely negatively, we run into the following conjecture.

\begin{Cj} 
Every infinite group $G$ of cardinality $\varkappa$ can be partitioned $G=\bigcup\limits_{n<\w}A_n$ so that $cov(A_nA_n^{-1})=\varkappa$ for each $n\in\w$. 
\end{Cj}

Now we confirm Conjecture for every group of regular cardinality and note that it holds also for some groups 
(in particular, Abelian) of an arbitrary cardinality (see \cite{b41}).

\begin{Th}\label{t5.2}
Let $G$ be an infinite group of cardinality $\varkappa$. Then there exists a partition 
$G=\bigcup\limits_{n\in\w}A_n$ such that $cov(A_nA_n^{-1})\ge cf(\kappa)$ for each $n\in\w$.
\end{Th}
\begin{proof}
If $G$ is countable, the statement is trivial: take any partition of $G$ into finite subsets. 
We assume that $|G|>\aleph_0$ and take a filtration $\{G_\alpha:\alpha<\kk\}$ of $G$ such that $|G_\alpha|<\kk$
for each $\alpha < \kk$. For $g\in G\setminus\{e\}$, we use the canonical representation 
$$g=x_{\gamma_{s(g)}(g)}x_{\gamma_{s(g)-1}(g)}\dots x_{\gamma_2(g)}x_{\gamma_1(g)}.$$

Each ordinal $\alpha<\varkappa$ can be written uniquely as $\alpha=\beta+n$ where $\beta$ is a limit ordinal and $n\in\w$. 
We put $f(\alpha)=n$ and denote by $Seq(\w)$ the set of all finite sequences of elements of $\w$. 
Then we define a mapping $\chi:G\setminus\{e\}\to Seq(\w)$ by 
$$\chi(g)=f(\gamma_{s(g)}(g))f(\gamma_{s(g)-1}(g))...f(\gamma_2(g))f(\gamma_1(g)),$$
and, for each $s\in Seq(\w)$, put $H_s=\chi^{-1}(s)$. Since the set $Seq(\w)$ is countable, it suffices to prove that $cov(H_sH_s^{-1})\ge cf\varkappa$ for each $s\in Seq(\w)$.

We take an arbitrary $s\in Seq(\w)$ and an arbitrary $K\subseteq G$ such that $|K|< cf\varkappa$. 
Then we choose $\gamma<\varkappa$ such that $\gamma>\max g$ for each $g\in K$ and $f(\gamma)\notin\{s_1,...,s_n\}$. 
We pick $h\in X_\gamma$ and show that $KH_s\cap hH_s=\varnothing$ so $h\notin KH_sH_s^{-1}$ and $cov(H_sH_s^{-1})\ge cf\kk$.

If $g\in KH_s$ and $\gamma_i(g)\ge\gamma$ then $f(\gamma_i(g))\in\{s_1,...,s_n\}$. If $g'\in hH_s$ then, by the 
choice of $h$, there exists $j$ such that $\gamma_j(g')=\gamma$ and $f(\gamma_j(g'))\notin\{s_1,\dots, s_n\}$.
Hence $\chi(g)\neq\chi(g')$ and $KH_s\cap hH_s=\varnothing$.
\end{proof}

The first usage of filtration was in \cite{b28} to partition every infinite group $G$ into $\aleph_0$ small subsets.
Let $D_n=\{g\in G:s(g)=n\}$. Then each subset $D_n$ is left small so $D_n\cap D_m^{-1}$ is small for all $n,m\in\NN$.
Thus $\{e\}\cup\bigcup_{n,m}D_n\cap D_m^{-1}$ is a partition of $G$ into small subsets.
\begin{Th}\label{t5.3}
Every infinite group $G$ can be partitioned into $\aleph_0$ scattered subsets.
\end{Th}
{\em Sketch of proof.} We use induction by the cardinality of $G$.
If $G$ is countable, the statement is evident because each singleton is scattered.
Assume that we have proved the theorem for all groups of cardinality $<\kk$ ($\kk>\aleph_0$) and take an  arbitrary
group of cardinality $\kk$. 
We fix a filtration $\{G_\alpha:\alpha<\kk\}$ of $G$ such that $|G_\alpha|<\kk$ for each $\alpha<\kk$.

For every $\alpha<\kk$, we use the inductive hypothesis to define a mapping $\chi_\alpha:G_{\alpha+1}\setminus 
G_{\alpha}\to\NN$ such that $\chi_\alpha^{-1}(i)$ is scattered in $G_{\alpha+1}$ for every $i\in\NN$.
Then we take $g\in G\setminus\{e\}$, $g=x_{\alpha_n}\dots x_{\alpha_1}$ and put 
$$\chi(g)=(\chi_{\alpha_n}(x_{\alpha_n}),\chi_{\alpha_{n-1}}(x_{\alpha_n}x_{\alpha_{n-1}}),\dots,\chi_{\alpha_1}(x_{\alpha_n}\dots x_{\alpha_1})).$$
Thus, we have defined a mapping $\chi:G\setminus\{e\}\to\bigcup_{n\in\NN}\NN^n$.
For verification that each subset $\chi^{-1}(m)$, $m\in \bigcup_{n\in\NN}\NN^n$ is scattered see \cite{b5}.

{\em Comments.} In the preprint form, Theorems~\ref{t5.1},~\ref{t5.2} and~\ref{t5.3} 
appeared in \cite{b40}, \cite{b41}, \cite{b5} respectively.

\section{Submeeasures}\label{s6}

Let $X$ be a set. A function $\mu:\PP_X\to[0,1]$ is called
\begin{itemize}
\item{} {\it density} if $\mu(\varnothing)=0$, $\mu(X)=1$ and $A\subseteq B$ implies $\mu(A)\le\mu(B)$;
\item{} {\it submeasure} if $\mu$ is density and $\mu(A\cup B)\le\mu(A)+\mu(B)$ for any subsets $A,B$ of $G$;
\item{} {\it measure} if $\mu$ is density and $\mu(A\cup B)=\mu(A)+\mu(B)$ for any disjoint subsets $A,B$ of $G$.
\end{itemize}

For our goals, the basic papers are \cite{b18}, \cite{b46}, \cite{b49}. 
In  \cite{b46} Solecki introduced three densities on a group $G$ by the formulas:
$$\sigma^L(A)=\inf_{F\in[G]^{<\w}}\sup_{x\in G}\frac{|F\cap xA|}{|F|}\text{, }\sigma^R(A)=\inf_{F\in[G]^{<\w}}\sup_{x\in G}\frac{|F\cap Ax|}{|F|}$$
$$\sigma(A)=\inf_{F\in[G]^{<\w}}\sup_{x,y\in G}\frac{|F\cap xAy|}{|F|}$$ 
and proved that $\sigma$ is a submeasure for every group $G$ and $\sigma^L$, $\sigma^R$ are submeasures provided that $G$ is amenable.

Recently \cite{b1} Banach modernized the Solecki densities to a couple of so-called extremal densities and applied 
them to some combinatorial problems.

{\em Given any finite partition of a group $G=A_1\cup\dots\cup A_n$, do there exists a cell $A_i$ and a subset $F$
of $G$ such that $G=FA_iA_i^{-1}$ and $|F|\le n$? This is so if $G$ is amenable or $n=2$.}

This is Problem $13.44$ from the Kourovka notebook \cite{b23} posed by the first author in $1995$.
For nowday state of this open problem see the survey \cite{b6}. Here we formulate some results from \cite{b6}.

{\em For every group $G$ and any partition $G=A_1\cup\dots\cup A_n$, there exists cells $A_i$, $A_j$ such that
\begin{itemize}
\item[(1)] $G=FA_jA_j^{-1}$ for some set $F\subseteq G$ of cardinality $|F|\le\max_{0<k<n}\sum\limits_{p=0}^{n-k}k^p\le n!$;
\item[(2)] $G=F(xA_iA_i^{-1}x^{-1})$ for some finite sets $F,E$ of $G$, $|F|\le n$;
\item[(3)] the set $(A_iA_i^{-1})^k$ is a subgroup of index $\le n$ in $G$;
\item[(4)] $G=FA_iA_i^{-1}A_i$ for some subset $F$ of cardinality $\le n$.
\end{itemize}}

The statements $(2)$, $(3)$ were proved with usage of the extremal densities.
We note that $(2)$ solves Problem $13.44$ provided that each cell of the partition is inner-invariant.

The family of left small subset of a group $G$ is an ideal in $\PP_G$.
Hence, for every finite partition $G$, at least one cell $A$ of the partition is not left small, so $A$ is left prethick.

{\em Given an infinite group $G$, does there exist a natural number $k=k(G)$ such that, for any partition 
$G=A_1\cup A_2$, there exists $i$ and $F\in[G]^{<\aleph_0}$ such that $FA_i$ is left thick and $|F|\le k(G)$?}

This question was asked in \cite{b36} and answered negatively of all Abelian groups, all countable locally finite
groups and all countably residually finite group.

For every countable group, the negative answer were obtained in \cite{b7} with help of syndedic submeasures.

A left invariant submeasure $\mu$ on $G$ is called {\em syndedic} if, for every subset $A\subset G$ with $\mu(A)<1$
and each $\ee>\frac{1}{|G|}$, there is a left large subset $L\subseteq G\setminus A$ such that $\mu(L)<\ee$.
The existence of syndedic submeasure on each countable group follows from \cite{b47}.
We don't know if a syndedic submeasure exists on each uncountable group.

\section{Three Sets Lemma}\label{s7}

Let $X$ be a set, $f:X\to X$. We define a directed graph $\Gamma_f$ with the set of vertices $X$ and the set of edges
$\{(x,y):y=f(x)\}$. Clearly, each vertex $x$ of $\Gamma_f$ is incident to only one edge $(x,y)$ and every connected
component of $\Gamma_f$ has only one directed cycle. 
We take four colors $1,2,3,4$ and color all fixed points of $X$ in $1$.
Then all other points can be easily colored in $2,3,4$ so that all incident points $x,y$ are of distinct colors.
If $\Gamma_f$ has no odd cycles of length $>1$, it suffices to use only three colors.
Thus, we have got the following statements
\begin{itemize}
\item[(1)] {\em A set $X$ can be partitioned $X=X_0\cup X_1\cup X_2\cup X_3$ so that $X_0=\{x\in X:f(x)=x\}$ and 
$f(X_i)\cap X_i=\varnothing$, $i\in\{1,2,3\}$.}
\item[(2)] {\em If $\Gamma_f$ has no odd cycles of length $>1$, then $X$ can be partitioned $X=X_0\cup X_1\cup X_2$ 
so that $X_0=\{x\in X:f(x)=x\}$ and $f(X_i)\cap X_i=\varnothing$, $i\in\{1,2\}$.}
\end{itemize}

If $f$ has no fixed points, the statement $(1)$ is known as $3$-sets lemma.
For its application to ultrafilters see \cite{b14}, \cite{b17}.
We consider only one easy application of this lemma to our subject, for others more delicate see \cite{b48}.

{\em Every group $G$ can be partitioned $G=X_1\cup X_2\cup X_3$ so that each subset $X_i$ is not left thick.
If $G$ has an element of even or infinite order then there is such a partition in two subsets.}

We take an arbitrary element $g\in G\setminus\{e\}$, consider a mapping $f:G\to G$ defined by $f(x)=gx$ and apply $(1)$.
In the second case we take $g$ of infinite or even order.
In both cases we have $\{e,g\}x\nsubseteq X_i$ for each $x\in X_i$.
Since $g\notin X_iX_i^{-1}$, we have $G\neq X_iX_i^{-1}$. On the other hand \cite{b16}, if each element of $G$ is of
odd order then, for every partition $G=A\cup B$, either $G=AA^{-1}$ or $G=BB^{-1}$.

\section{Partitions into thick subsets}\label{s8}
\begin{Th} \label{t8.1} 
Let $G$ be an infinite group of cardinality $\kk$, $\lambda$ be an infinite cardinal, $\lambda\le\kk$. 
Then $G$ can be partitioned into $\kk$ $\lambda$-thick subsets provided that either $\lambda=\kk$ and $\kk$ is regular or $\lambda<\kk$.
\end{Th}
\begin{proof}
We consider only the first case: $\lambda=\kk$ and $\kk$ is regular.
We write $G$ as union of an increasing chain of subsets $\{H_\alpha:\alpha<\kk\}$, $|H_\alpha|<\kk$, $H_0=\{e\}$.
Then we choose inductively a $\kk$-sequence $(x_\alpha)_{\alpha<\kk}$ in $G$ such that the subsets 
$\{H_\alpha x_\alpha H_\alpha:\alpha<\kk\}$ are pairwise disjoint.
We partition $\kk=\bigcup_{\gamma<\kk}K_\gamma$ into $\kk$ cofinal subsets and put 
$A_\gamma=\bigcup\{H_\alpha x_\alpha H_\alpha:\alpha\in K_\gamma\}$. It suffices to show that each $A_\gamma$ is 
$\kk$-thick. We take an arbitrary $Y\in[G]^{<\kk}$. Since $\kk$ is regular, there is $\alpha\in K_\gamma$ such that 
$Y\subseteq H_\alpha$. Then $Yx_\alpha\subseteq A_\gamma$ and $x_\alpha Y\subseteq A_\gamma$.
\end{proof}

Surprisingly, the singular case is "cardinally" different.
\begin{Th}\label{t8.2} For every group $G$ of singular cardinality $\kk$, the following statements hold
\begin{itemize}
\item[{(i)}] if a subset $A$ of $G$ is left $\kk$-thick then $A$ is right $\kk$-large;
\item[{(ii)}] $G$ cannot be partitioned into two $\kk$-thick subsets;
\item[{(iii)}] if $G$ is Abelian then, for every finite partition $G=A_1\cup\dots\cup A_n$, at least one cell $A_i$ is $\kk$-large.
\end{itemize}
\end{Th}
\begin{proof} $(i)$ We write $G$ as a union $G=\bigcup\{H_\alpha:\alpha<cf\kk\}$ of subsets from $[G]^{<\kk}$.
For each $\alpha<cf\kk$, we pick $g_\alpha\in A$ such that $H_\alpha g_\alpha\subseteq A$ so $H_\alpha\subseteq Ag^{-1}_\alpha$.
We put $F=\{g_\alpha^{-1}:\alpha<cf\kk\}$. Then $F\in[G]^{<\kk}$ and $G=\bigcup\{H_\alpha:\alpha< cf\kk\}\subseteq AF$.
Hence, $A$ is right $\kk$-large.

$(ii)$ Suppose that $G$ is partitioned $G=A\cup B$ such that $A$ is $\kk$-thick.
Then $A$ is left $\kk$-thick and, by $(i)$ $A$ is right $\kk$-large. 
Hence, $G\setminus A$ is not right $\kk$-thick and $B$ is not $\kk$-thick.

$(iii)$ We proceed by induction. For $n=1$ the statement is evident. Let $G=A_1\cup\dots\cup A_{n+1}$.
We put $B=A_2\cup\dots\cup A_{n+1}$. If $A_1$ is not large then $B$ is thick.
By $(i)$, $B$ is large so $G=FB$ for some $B\in[G]^{<\kk}$. 
Since $G=FA_2\cup\dots\cup FA_{n+1}$, by the inductive hypothesis, there exists $i\in\{2\dots n+1\}$ such that $FA_i$ is large.
Hence $A_i$ is large.
\end{proof}

By Theorem~\ref{t8.1}, every infinite group $G$ of regular cardinality $\kk$ can be partitioned $G=B_1\cup B_2$ so that each subset $B_i$ is not left $\kk$-large.
In \cite[Prolem 13.45]{b23}, the first author asked if the same is true for every group $G$ of singular cardinlity $\kk$.
Theorem~\ref{t8.2}$(iii)$ gives a negative answer for every Abelian group of singular cardinality $\kk$.
On the other hand, this is so for every free group.
\begin{Th}\label{t8.3} 
Every free group $F_A$ in the alphabet $A$ can be partitioned $F_A= B_1\cup B_2$ so that each cell $B_i$ is not left $\kk$-large. 
\end{Th}
\begin{proof} In view of Theorem~\ref{t8.1}, we may suppose that $|A|>\aleph_0$, so $|A|=|F_A|$. We partition $A=A_1\cup A_2$ so that $|A|=|A_1|=|A_2|$, and put
$$B_1=\{g\in F_A:\rho(g)\in A_1\cup A_1^{-1}\},\text{ } B_2=F_A\setminus B_1.$$
Assume that $F_A=HB_1$ for some $H\in[F_A]^{<\kk}$, where $\kk=|F_A|$.
Then we choose $c\in A_2$ such that $c$ and $c^{-1}$ do not occur in any $h\in H$.
Clearly, $c\notin HB_1$ and $B_1$ is not left $\kk$-large. Analogously, $B_2$ is not left $\kk$-large.
\end{proof}

For applications of above results, we need some definitions.

A topological space $X$ with no isolated points is called {\em maximal} if $X$ has an isolated point in any stronger topology.
We note that $X$ is maximal if and only if, for every point $x\in X$, there is only one free ultrafilter converging to $x$.

A topology $\tau$ on a group $G$ is called {\em left invariant} if, for every $g\in G$, the left shift $x\mapsto gx:\text{ }G\to G$ is continuous in $\tau$.
A group $G$ endowed with a left invariant topology is called {\em left topological}. 
We say that a left topological group $G$ is maximal if $G$ is maximal as a topological space.
By \cite[$\S$2]{b26}, every infinite group $G$ of cardinality $\kk$ admits $2^{2^\kk}$ distinct maximal left invariant topologies.

A left topological group $G$ of cardinality $\kk$ is called $\kk$-bounded if each neighborhood $U$ of the identity $e$ (equivalently, each non-empty open subset of $G$) is left $\kk$-large.
We note that each left $\kk$-thick subset of $G$ meets every left $\kk$-large subset of $G$. 
It follows that $A$ is dense in every $\kk$-bounded topology on $G$.
Hence, if $G$ can be partitioned into two left $\kk$-thick subsets then every $\kk$-bounded topology on $G$ is not maximal.
Applying Theorems ~\ref{t8.1} and ~\ref{t8.3}, we get the following theorems.

\begin{Th} \label{t8.4}
An infinite group $G$ of cardinality $\kk$ admits no maximal left invariant $\kk$-bounded topology provided that 
either $\kk$ is regular or $G$ is a free group. 
\end{Th}

The following theorem answers affirmatively Question 4.4 from \cite{b26}.
\begin{Th}\label{t8.5} 
Every Abelian group $G$ of singular cardinality $\kk$ admits a maximal left invariant $\kk$-bounded topology. 
\end{Th}
\begin{proof}
We use a technique from \cite{b17}: endow $G$ with the discrete topology, identify the Stone-$\check{C}$ech compactification 
$\beta G$ of $G$ with the set of all ultrafilters on $G$ and consider $\beta G$ as a right topological semigroup. We put
$$L=\{q\in\beta G:\text{ each member } Q\in q\text{ is left }\kk\text{-large}\}.$$
Applying Theorem~\ref{t8.2}$(iii)$ and Theorem 5.7 from \cite{b17}, we conclude that $L\neq\varnothing$. 
It is easy to see that $L$ is a closed subsemigroup of $\beta G$ so $L$ has some idempotent $p$.
The ultrafilter $p$ defines a desired topology on $G$ in the following way:
for each $g\in G$, the family $\{gP\cup\{g\}:P\in p\}$ forms a base of neighborhoods of $g$.
\end{proof}

{\em Comments.} For $\lambda=\aleph_0$, Theorem~\ref{t8.1} was proved in \cite{b22} to partition each infinite
totally bounded topological group $G$ into $|G|$ dense subsets. The results of this section are from preprint \cite{b42}.
More on partitions of topological groups into dense subsets see \cite{b25} and \cite[Chapter 13]{b14}.

\section{Partitions into thin subsets}\label{s9}

Given an infinite group $G$, we denote by $\mu(G)$ and $\eta(G)$ the minimal cardinals such that 
$G$ can be partitioned into $\mu(G)$ and $\eta(G)$ thin and sparse subsets respectively.

\begin{Th}\label{t9.1} 
For a group $G$, $\mu(G)=|G|$ if $|G|$ is a limit cardinal and $\mu(G)=\gamma$ if $|G|=\gamma^+$.
\end{Th}

\begin{Th}\label{t9.2} 
Let $G$ be a group, $\kk$ be an infinite cardinal. If $|G|>(\kk^+)^{\aleph_0}$ then $\eta(G)>\kk$. 
If $|G|\le\kk^+$ then $\eta(G)\le\kk$.
\end{Th}

In particular, if $|G|>2^\kk$ then $\eta(G)>\kk$, if $\eta(G)=\aleph_0$ then $\aleph_0\le|G|\le2^{\aleph_0}$.
\begin{Qs} Does $|G|=2^{\aleph_0}$ imply $\eta(G)=\aleph_0$? \end{Qs}

Under CH, Theorem~\ref{t9.2} gives an affirmative answer to this question.
To answer the question negatively under $\neg$CH, it suffices to show that, for any $\aleph_0$-coloring of
$\aleph_2\times\aleph_2$, there is a monochrome subset $A\times B$, $A,B\subset\aleph_2$, $|A|=|B|=\aleph_0$.

Given a group $G$ and a subset $A$ of $G$, how one can detect whether $A$ belongs to the ideal $<\tau_G>$ in $\PP_G$
generated by thin subsets. To answer this question we need some test which, for given $A\subset G$ and $m\in\NN$,
detects whether $A$ can be partitioned in $\le m$ thin subsets.

Remind that a subset $A$ is $m$-thin if, for each $F\in[G]^{<\aleph_0}$, there is $K\in[G]^{<\aleph_0}$ such that
$|Fg\cap A|\le m$ for each $g\in G\setminus K$. It is easy to see that a union of $\le m$ thin subsets is $m$-thin.
If $G$ is countable, the converse statement holds \cite{b20}, so $<\tau_G>=\bigcup_{m\in\NN}\tau_{G,m}$ where 
$\tau_{G,m}$ is a family of all $m$-thin subsets of $G$.

{\em Can every $m$-thin subset of an arbitrary uncountable group be partitioned in $\le m$ subsets?}

Answering this question from \cite{b20}, G. Bergman constructed a group $G$ of cardinality $\aleph_2$ and a $2$-thin
subset $A$ of $G$ which cannot be partitioned into two thin subsets.
His remarkable construction can be found in \cite{b37} and based on the following combinatorial claim.

For a group $G$ and $g\in G$, we say that $G\times\{g\}$ is a horizontal line in $G\times G$, $\{g\}\times G$ is a 
vertical line in $G\times G$, $\{(x,gx):x\in G\}$ is a diagonal in $G\times G$.
For a group $G$ with the identity $e$, the following statements hold

\begin{itemize}
\item[(1)] if $|G| \geqslant \aleph_2$ and $\chi: G\times G \to \{1,2,3\} $ then there is $g\in G$, $g\ne e$ such 
that either some horizontal line $G\times \{g\}$ has infinitely many points of color 1, or some vertical line 
$\{g\}\times G$ has infinitely many points of color 2, or some diagonal $\{(x,gx): x\in G\}$ has infinitely many 
points of color 3;

\item[(2)]  if $|G| \leqslant \aleph_1$ then there is a coloring $\chi: G\times G \to \{1,2,3\} $ such that each 
horizontal line has only finite number of points of color 1, each vertical line has only finite number of points of 
color 2, each diagonal has only finite number of points of color 3.
\end{itemize}

\begin{Th}\label{t9.3}
For each natural number $m\ge2$, there exist an Abelian group $G$ of cardinality $\aleph_n$, $n=\frac{m(m+1)}{2}-1$,
and a $2$-thin subset $A$ of $G$ which cannot be partitioned into $m$ thin subsets.
\end{Th}

\begin{Th}\label{t9.4}
There are an Abelian group $G$ of cardinality $\aleph_\w$ and a $2$-thin subset $A$ of $G$ which cannot be
partitioned into $m$ thin subsets for each $m$ in $\NN$.
\end{Th}

\begin{Th}\label{t9.5}
Every $m$-thin subset of an Abelian group $G$ of cardinality $\aleph_n$ can be partitioned into $\le m^{n+1}$ thin subsets.
\end{Th}

By Theorem~\ref{t9.5}, $<\tau_g>=\bigcup_{m\in\NN}\tau_{G,m}$ for each Abelian group of cardinality $<\aleph_\w$, 
and, by Theorem~\ref{t9.4}, this statement does not hold for some Abelian groups of cardinality $\ge\aleph_\w$.

\begin{Qs} 
Can each $m$-thin subset of an Abelian group $G$ of cardinality $\aleph_1$ be partitioned into $\le m$ thin subsets? 
\end{Qs}

{\em Comments.} 
Theorems~\ref{t9.1} and~\ref{t9.2} are from \cite{b29} and \cite{b30}, the remaining results from \cite{b37}.

We remind that a subset $A$ of a group $G$ is thin if and only if $G$ is thin in the ballean $\BB_l(G,\aleph_0)$.
For partitions of a group into thin subsets in the balleans $\BB_l(G,\kk)$ for every infinite cardinal $\kk$ see \cite{b29}.

In \cite{b13}, Erdos and Kakutani proved that the Continuum Hypothesis is equivalent to the following combinatorial
statement: $\RR\setminus\{0\}$ can be partitioned in $\aleph_0$ subsets linearly independent over $\QQ$.
In \cite{b4}, the authors use thin partitions to reprove this theorem and extend it to some matroids.

\section{Miscellaneous}\label{s10}

{\bf 1. Partitions of free groups.} This subsection is about partitions of free group into large subsets, it 
contains two theorems from preprint \cite{b40}.

For a cardinal $\kk$, we denote by $F_\kk$ the free group in the alphabet $\kk$.
Given any  $g\in F_\kk\setminus\{e\}$ and $a\in\kk$, we write $\lambda(g)=a$ ($\rho(g)=a$) if the first (the last)
letter in the canonical representation of $g$ is either $a$ or $a^{-1}$.

\begin{Th}\label{t10.1} 
For any infinite cardinal $\kk$, the following statements hold

$(i)$ $F_\kk$ can be partitioned into $\kk$ left $3$-large subsets;

$(ii)$ $F_\kk$ can be partitioned into $\kk$ $4$-large subsets.
\end{Th}
\begin{proof}
$(i)$ For each $a\in\kk$, we put $P_a=\{g\in F_\kk\setminus\{e\}:\lambda(g)=a\}$ and note that $F_\kk=\{e,a\}P_a$.

$(ii)$ We partition $\kk$ into $2$-element  subsets $\kk=\bigcup_{\alpha<\kk}\{x_\alpha,y_\alpha\}$ and put $X=\{x_\alpha:\alpha<\kk\}$, $Y=\{y_\alpha:\alpha<\kk\}$.

For every $\alpha<\kk$, we denote 
$$L_\alpha=\{g\in F_\kk\setminus\{e\}:\lambda(g)=x_\alpha\Leftrightarrow \rho(g)\in X, \lambda(g)=y_\alpha\Leftrightarrow \rho(g)\in Y,\}$$
$$R_\alpha=\{g\in F_\kk\setminus\{e\}:\rho(g)=x_\alpha\Leftrightarrow \lambda(g)\in Y, \rho(g)=y_\alpha\Leftrightarrow \rho(g)\in X,\}$$
Then we put $P_\alpha=L_\alpha\cup R_\alpha$ and note that the subsets $\{P_\alpha:\alpha<\kk\}$ are pairwise disjoint.

Given any $g\in F_\kk$, we have
$${\{e,x_\alpha,y_\alpha\}g\cap L_\alpha\neq\varnothing},\text{ }g\{e,x_\alpha,y_\alpha\}\cap R_\alpha\neq\varnothing.$$
Hence, $L_\alpha$ is left $4$-large and $R_\alpha$ is right $4$-large, so $P_\alpha$ is $4$-large.
\end{proof}

\begin{Th}\label{t10.2}
For a natural number $n\ge2$, the following statements hold

$(i)$ $F_n$ can be partitioned into $\aleph_0$ left $3$-large subsets;

$(ii)$ $F_n$ can be partitioned into $\aleph_0$ $5$-large subsets.
\end{Th}
\begin{Qs}
For an infinite cardinal $\kk$, can the free group $F_\kk$ be partitioned into $\kk$ $3$-large subsets?
\end{Qs}
\begin{Qs}
For a natural number $n\ge2$, can the free group $F_n$ be partitioned into $\aleph_0$ $4$-large subsets?
\end{Qs}

{\bf 2. Partitions of almost $P$-small subsets.} A subset $A$ of a group $G$ is called
\begin{itemize}
\item{} {\em $P$-small} if there exists an injective sequence $(g_n)_{n\in\w}$ in $G$ such that the subsets 
$\{g_nA:n\in\w\}$ are pairwise disjoint;
\item{} {\em almost $P$-small} if there exists an injective sequence $(g_n)_{n\in\w}$ in $G$ such that $g_nA\cap g_mA$
is finite for all distinct $m,n\in\w$.
\end{itemize}

{\em Every almost $P$-small subset of a group $G$ can be partitionwd into two $P$-small subsets.}

This unexpected theorem was proved in \cite[Theorem 3.5]{b19}.

{\bf 3. Packing and covering numbers.} 
For a subset $A$ of a group $G$, the packing and covering number of $A$ are defined as
\begin{itemize}
\item $pack(A)=\sup\{|S|:S\subseteq G,\{sA:s\in S\} \text{ is disjoint}\}$,
\item $cov(A)=\min\{|S|:S\subseteq G, G=SA\}$.
\end{itemize}

For some results and open problems concerning these numbers see \cite{b32}.

{\bf Kaleidoscopical configurations.}
Let $G$ be a group and let $X$ be a $G$-space endowed with the action $G\times X\to X,\text{ }(g,x)\mapsto gx$.
A subset $A$ is called a {\em kaleidoscopical configuration} if there exists a coloring $\chi:X\to|A|$ such that
$\chi|_{gA}$ is bijective for each $g\in G$.

For kaleidoscopical configurations see recent survey \cite{b43}.
The most intriguing open problem in this area: does there exist a finite kaleidoscopical configuration $K$, $|K|>1$
in $\RR^2$, where the plane is considered as a $G$-space under the action of the group of all isometries of $\RR^2$.
We know only that such a $K$ must contain (if exists) at least $5$ points.

Igor Protasov (i.v.protasov@gmail.com)\\
Department of Cybernetics, Kyiv University,\\
Volodymyrska 64, Kyiv 01033, Ukraine \\\\
Sergii Slobodianiuk (slobodianiuk@yandex.ru)\\
Department of  Mathematics and Mechanics, Kyiv University,\\
Volodymyrska 64, Kyiv 01033, Ukraine
\end{document}